\documentclass[12pt,leqno, english]{amsart}
\title[Kunneth formula for pullbacks]
{A Kunneth formula for Bredon cohomology of pullbacks and Twisted K-theory of some 6-dimensional orbifolds.}
\author{German Combariza}
\address{Departamento de Matem\'aticas. \\Pontificia Universidad Javeriana\\Cra. 7 No. 43-82 - Edificio Carlos Ort\'iz 5to piso\\ Bogot\'a D.C, Colombia}

\email{germancombariza@javeriana.edu.co}

\author{Mario Vel\'asquez}
\address{Departamento de Matem\'aticas. \\Pontificia Universidad Javeriana\\Cra. 7 No. 43-82 - Edificio Carlos Ort\'iz 5to piso\\ Bogot\'a D.C, Colombia}

\email{mario.velasquez@javeriana.edu.co}
 \urladdr{https://sites.google.com/site/mavelasquezm/}
         \date{\today}
         
\usepackage{amsmath}
\usepackage{hyperref}
\usepackage{verbatim}
\usepackage{pdfsync}
\usepackage{calc}
\usepackage{enumerate,amssymb}
\usepackage[arrow,curve,matrix,tips,2cell]{xy}
\SelectTips{eu}{10} \UseTips
\UseAllTwocells
\usepackage{graphicx}
\usepackage{pb-diagram}

\input xy
\xyoption{all}

\DeclareMathAlphabet\EuR{U}{eur}{m}{n}
\SetMathAlphabet\EuR{bold}{U}{eur}{b}{n}

\makeindex             

\usepackage{pstricks}
\usepackage{pst-3dplot}

\theoremstyle{plain}
\newtheorem{theorem}{Theorem}[section]
\newtheorem{lemma}[theorem]{Lemma}
\newtheorem{proposition}[theorem]{Proposition}
\newtheorem{corollary}[theorem]{Corollary}

\theoremstyle{definition}
\newtheorem{definition}[theorem]{Definition}
\newtheorem{example}[theorem]{Example}

\newtheorem{remark}[theorem]{Remark}

{\catcode`@=11\global\let\c@equation=\c@theorem}

\newcommand{\comsquare}[8]                   
{\begin{CD}
		#1 @>#2>> #3\\
		@V{#4}VV @V{#5}VV\\
		#6 @>#7>> #8
	\end{CD}
}

\newcommand{\xycomsquare}[8]                   
{\xymatrix
	{#1 \ar[r]^{#2} \ar[d]^{#4} &
		#3 \ar[d]^{#5}  \\
		#6\ar[r]^{#7} &
		#8
	}
}

\newcommand{\calf}{\mathcal{F}}

\newcommand{\calfin}{\mathcal{FIN}}

\newcommand{\calM}{\mathcal{M}}
\newcommand{\calN}{\mathcal{N}}

\newcommand{\caln}{{\mathcal N}}

\newcommand{ \calr}{\mathcal{R}}
\newcommand{\calm}{{\mathcal M}}

\newcommand{\IC}{{\mathbb C}}

\newcommand{\IR}{{\mathbb R}}

\newcommand{\IZ}{{\mathbb Z}}

\newcommand{\curs}{\EuR}

\newcommand{\MODULES}{\curs{MODULES}}
\newcommand{\Or}{\curs{Or}}

\newcommand{\ch}{\underline{C}}

\newcommand{\colim}{\operatorname{colim}}

\newcommand{\charr}{\operatorname{Ch}}

\newcommand{\Ext}{\operatorname{Ext}}

\newcommand{\GL}{\operatorname{GL}}

\newcommand{\Hom}{\operatorname{Hom}}

\newcommand{\PGL}{\operatorname{PGL}}

\newcommand{\Tor}{\operatorname{Tor}}

\newcommand{\OrGF}[2]{\Or_{#2}(#1)}                

\newcommand{\higherlim}[3]{{\setbox1=\hbox{\rm lim}
		\setbox2=\hbox to \wd1{\leftarrowfill} \ht2=0pt \dp2=-1pt
		\mathop{\vtop{\baselineskip=5pt\box1\box2}}
		_{#1}}^{#2}#3}

\newcommand{\version}[1]                       
{\begin{center} last edited on #1\\
		last compiled on \today \\
		name of texfile: \jobname
	\end{center}
}

\newcounter{commentcounter}

\makeindex

\begin{document}
\maketitle	
	\begin{abstract}In this paper we prove a Kunneth formula for Bredon cohomology for actions of a  pullback of groups. We show how this formula can be used to compute orbifold twisted K-theory for some discrete twistings. Using that result, we compute orbifold K-theory for some 6-dimensional orbifolds introduced by Vafa and Witten. These examples also show the limitations of the method.
		\end{abstract}
	
\section*{Introduction}
One of the useful tools in singular cohomology is the Kunneth formula. Let $X$ and $Y$ be topological spaces, it allows to compute $H^*(X\times Y)$ given knowledge of $H^*(X)$ and $H^*(Y)$. An interesting problem is to extend that result to extraordinary cohomology theories. A remarkable case is equivariant K-theory, introduced by Segal in \cite{segal1} and defined for proper actions of discrete groups in \cite{luck-cherncharacters}. Let $\Gamma$ be a discrete group acting properly on spaces $X$ and $Y$, one wants to compute$$K_G^*(X\times Y)$$in terms of $K_G^*(X)$ and $K_G^*(Y)$. In this case, to obtain an analogue to the Kunneth formula is an open problem for actions of discrete groups (even for finite groups). For actions of $\IZ/2\IZ$ a Kunneth theorem was proved in \cite{rosenberg}.

An approximation of that problem is to consider Bredon cohomology. When we take coefficients in the Bredon module of representations (see Def. \ref{repfunctor}) there is an spectral sequence converging to equivariant K-theory whose $E^2$-term correspond Bredon cohomology.

Let   $\Gamma$ be  a discrete group which   is obtained as a pullback diagram
 \begin{equation}\label{diagrampullback}
 \xymatrix{
 	\Gamma \ar[r]^{p_2}\ar[d]_{p_1} & H\ar[d]^{\pi_2}\\
 	G \ar[r]_{\pi_1} & K}
\end{equation}
When the maps $\pi_1$ and $\pi_2$ are clear from the context we denote this pullback by $G\times_KH$.

If $X$ is a proper $G$-CW-complex and $Y$ is a $H$-CW-complex, the space $X\times Y$ is naturally a $\Gamma$-CW-complex. In Theorem \ref{maintheorem} we provide a Kunneth formula computing $\Gamma$-equivariant Bredon cohomology of $X\times Y$ in terms of Bredon cohomology of $X$ and $Y$.

One of the advantages to work with pullbacks is that this includes some interesting cases. For example
\begin{itemize}
	\item If $K$ is the trivial group we obtain a Kunneth formula for the action of the direct product $G\times H$ over $X\times Y$.
	\item If $K=G=H$, we obtain a Kunneth formula for the diagonal action of $G$ over a product $X\times Y$.
\end{itemize}

On the other hand if we take Bredon cohomology with constant coefficients we obtain a Kunneth formula for group cohomology of pullbacks.

As our main application in Corollary \ref{y1} and Theorem \ref{twistedy2} we compute the orbifold K-theory and twisted orbifold K-theory (for an specific twisting) of some celebrated examples presented in \cite{vafa-witten}.

These orbifolds are 6-dimensional, this makes very difficult to compute their K-theories directly. Using our formula we obtain a simple calculation of it. These examples also show the limitations of our method (Remark \ref{limited}).

This paper continues the work started in \cite{BAJUVE} where some calculations concerning pullbacks of groups were made. We clarify some questions that appear in that work about the computation of twisted K-theory for pullbacks.

\section{Tensor product of Bredon modules}
In this section we introduce the notion of tensor product of Bredon modules over a Green functor.

Let $R$ be a commutative ring with unity. Let $\OrGF{G}{}$ be the orbit category of $G$; it is a category with objects $G/H$ for each subgroup $H\subseteq G$ and with morphisms given by $G$-equivariant maps where every $G$-equivariant map $G/H\mapsto G/K$ is multiplication by an element $g\in G$ such that $gHg^{-1}\subseteq K$. In a similar way one can also define the orbit category relative to a family of groups $\calf$, which we denote by $\OrGF{G}{\calf}$. For more information about $\OrGF{G}{\calf}$ the reader can consult \cite{luck1998}.


	Recall that a Bredon $R$-module over $\calf$ (or simply a Bredon module over $\calf$ if $R$ is clear from the context) is a contravariant functor $$\OrGF{G}{\calf}\to R-\MODULES.$$ A Mackey functor over $\OrGF{G}{\calf}$ is a bifunctor
$$\calM=(\calM_*,\calM^*):\OrGF{G}{\calf}\rightarrow R-\MODULES,$$satisfying certain double coset formula. For details consult \cite{webb}.
\begin{definition}
	A Green functor over $\OrGF{G}{\calf}$ is a Mackey functor $\caln$ together with a natural pairing $\caln\times\caln\to\caln$ such that for every $G/P\in\OrGF{G}{\calf}$ the map $\caln(G/P)\times \caln(G/P)\to \caln(G/P)$ makes $\caln(G/P)$ a commutative ring.
\end{definition}

\begin{example}
    Let  $R$ be a commutative ring with unity. The trivial functor $$\underbar{R}:\OrGF{G}{\calf}\to R-\MODULES$$ such that $G/H\mapsto R$ for every subgroup $H\in \calf$, is a Green functor. 
\end{example}


\begin{example}\label{repfunctor}
    Let $G$ be a discrete group. Let $\calfin_G$ be the family of finite subgroups of $G$ (or simply $\calfin$ if $G$ is clear from the context). The representation functor is a Bredon module
    $$ \mathcal{R}^G: \OrGF{G}{\calfin} \to \mathbb{Z}-\MODULES$$
    which associates to every object $G/H$ with $H\in\calfin$ its representation ring $R(H)$. The representation functor is a Green functor.    
\end{example}


Now we will define a notion of tensor product of Bredon modules over a Green functor. To define it we need to recall the notion of module over a Green functor.
\begin{definition}
	Let $\caln$ be a Green functor over $\OrGF{G}{\calf}$. An $\caln$-module consists of a Bredon module $\calm$ over $\OrGF{G}{\calf}$ and a natural transformation $\caln\times\calm\to\calm$	such that for every subgroup $P\in\calf$, each pairing morphism $$\caln(G/P)\times\calm(G/P)\to\calm(G/P)$$ endows to $\calm(G/P)$ with a unitary $\caln(G/P)$-module structure.
\end{definition}
\begin{definition}
	Let $\calf$ a family of subgroups of $G$. Let $f:G\to K$ be a group homomorphism, and let $\caln$ be a Green functor over $\OrGF{K}{f(\calf)}$ we define the Green functor $f^*\caln$ over $\OrGF{G}{\calf}$ as follows
	\begin{align*}
	f^*\caln:\OrGF{G}{\calf}&\to\IZ-\MODULES\\
	G/P& \to\caln(K/f(P)).
	\end{align*}
\end{definition}

\begin{example}\label{PullbackExampleBredon}
	Let $\Gamma$ be a group as in the diagram \eqref{diagrampullback}. The Bredon module $\calr^\Gamma$, is a $p_1^*\calr^G$-module, a  $p_2^*\calr^H$-module as well a $(\pi_1\circ p_1)^*\calr^K$-module. 
\end{example}
	\begin{definition}
		Let $\Gamma$ be a group coming from a diagram \eqref{diagrampullback}. Let $\calf$ be a family of subgroups of $K$, $\calf_1$ be a family of subgroups of $G$ and $\calf_2$ be a family of subgroups of $H$. We say that $(\calf,\calf_1,\calf_2)$ are a sequence of compatible families of subgroups for $\Gamma$ if $\pi_1(\calf_1)\subseteq\calf$ and $\pi_2(\calf_2)\subseteq\calf$. 
	\end{definition}
Let $(\calf,\calf_1,\calf_2)$ be a sequence of compatible families of subgroups of $\Gamma$, we denote by $\calf_1\times_\calf\calf_2$ to the family of subgroups of $\Gamma$ defined as
$$\calf_1\times_\calf\calf_2=\{P\times_{\pi_1(P)}Q\mid P\in\calf_1,\text{ and }Q\in\calf_2\}.$$

\begin{definition}Consider the diagram \eqref{diagrampullback}. Let $(\calf,\calf_1,\calf_2)$ be a sequence of compatible families of subgroups for $\Gamma$. Let $\mathcal{N}$ be a Bredon module over $\OrGF{K}{\calf}$, let $\mathcal{M}_1$ be a $(\pi_1^*\mathcal{N})$-module over $\OrGF{G}{\calf_1}$ and $\mathcal{M}_2$ be a $(\pi_2^*\mathcal{N})$-module over $\OrGF{H}{\calf_2}$. We define the tensor product $\mathcal{M}_1\otimes_\mathcal{N}\mathcal{M}_2$. It is the $(\pi_1\circ p_1)^*\mathcal{N}$-module over $\OrGF{\Gamma}{\calf_1\times_\calf\calf_2}$, satisfying
	\begin{align*}
	(\calM_1\otimes_\calN\calM_2):\OrGF{\Gamma}{\calf_1\times_\calf\calf_2}&\to \IZ-\MODULES\\
	\Gamma/(P\times_{\pi_1(P)}Q)&\mapsto \mathcal{M}_1(G/P)\otimes_{\mathcal{N}(K/\pi_1(P))}\mathcal{M}_2(H/Q)
	\end{align*}
\end{definition}
\section{Bredon cohomology associated to twisted K-theory}

In this section we briefly recall the definition of Bredon cohomology with coefficients in a Bredon module, then define twisted K-theory for discrete torsion and using the coefficients of this theory we define a Bredon module associated to twisted K-theory.
\subsection{Bredon cohomology}
 \begin{definition}Let $\calf$ be a family of subgroups of $G$. A $G$-CW complex is a CW complex with a $G$-action permuting the cells and such that if a cell is sent to itself, this is done by the identity map. A $\calf$-$G$-CW-complex is a $G$-CW-complex where all cell stabilizers are elements of $\calf$. A $\calfin$-$G$-CW-complex is called a proper $G$-CW-complex.
 \end{definition}

Let $X$ be a $\calf$-$G$-CW-complex. The cellular chain complex of $X$ is defined as the $\IZ$-graded Bredon module \begin{align*}\ch_n^G(X):\OrGF{G}{\calf}&\to\IZ-\MODULES\\\ch_n^G(X)(G/H)&=H_n((X^{(n)})^H,(X^{(n-1)})^H;\IZ)\end{align*}together with a boundary map $$\underline{\partial}:\ch_n^G(X)\to\ch_{n-1}^G(X)$$induced from the cellular boundary map, with $\underline{\partial}^2=0$. 

Consider two Bredon modules $\calM$ and $\calM'$, we denote by $$\Hom_{\OrGF{G}{\calf}}(\calM,\calM')$$ the abelian group of natural transformations from $\calM$ to $\calM'$. Let $X$ be a $\calf$-$G$-CW-complex and $\calM$ a Bredon module over $\OrGF{G}{\calf}$, we define the \emph{Bredon cochain complex with coefficients in $\calM$} as$$C^n_\calf(X;\calM)=\Hom_{\OrGF{G}{\calf}}(\ch_n(X),\calM)\text{, with }\delta=\Hom_{\OrGF{G}{\calf}}(\underline{\partial},id).$$
The homology of this complex is called the \emph{Bredon cohomology} of $X$ with coefficients in $\calM$ with respect to the family $\calf$, denoted by $H^*_{\calf}(X;\calM)$ or also by $H^*_G(X;\calM)$ when the family $\calf$ is clear from the context.
\subsection{Twisted K-theory}
Equivariant twisted K-theory was defined by Atiyah and Segal in \cite{atiyahsegaltwisted} for actions of compact Lie groups. The case of proper actions of discrete groups was treated in \cite{BEJU} for general twistings. On the other hand, the approach in \cite{ademruan} for orbifold twisted K-theory for discrete torsion has the advantage of a more concrete description using finite dimensional vector bundles. As is showed in \cite{dwyer2008}, the last approach can be adapted to the case of proper actions of discrete groups and discrete torsion. In this section we briefly recall this construction.

Let $V$ be a finite dimensional complex vector space, let $\GL(V)$ be the group of invertible linear transformations of $V$. Consider the central extension
\begin{equation}\label{pgl}	
0\to\IC^*\to\GL(V)\xrightarrow{\pi}\GL(V)/\IC^*\to0.
\end{equation}
We denote the quotient group in the above extension by $\PGL(V)$.
\begin{definition}
	A projective representation of $G$ is a pair $(\rho,V)$ where $V$ is a finite dimensional complex vector space and $$\rho:G\to\PGL(V)$$is a homomorphism.
\end{definition} 
Now if we choose a map (not a homomorphism) $$f:\PGL(V)\to\GL(V)$$ such that $\pi\circ f=id$ and take the pullback of the extension \ref{pgl}, we have a commutative diagram 
$$\xymatrix{0\ar[r]&\IC^*\ar[r]&\GL(V)\ar[r]_\pi&\PGL(V)\ar[r]\ar@{.>}@/_{5mm}/[l]_f&0  \\0\ar[r]&\IC^*\ar[u]_{id}\ar[r]&\rho^*G\ar[u]_{\bar{\rho}}\ar[r]_{\pi_1}&G\ar[u]_\rho\ar[r]&0}$$

Consider the map $\sigma=f\circ\rho$ and define
\begin{align*}
\alpha:G\times G&\to\IC\subseteq\GL(V)\\
(g,h)&\mapsto\alpha(g,h)=\sigma(g)\sigma(h)\sigma(gh)^{-1}.
\end{align*}
An straighforward calculation shows that $\alpha\in Z^2(G;\IC^*)$.
\begin{definition}
	Let $V$ be a complex vector space. Let us fix $\alpha\in Z^2(G;\IC^*)$. An $\alpha$-projective representation of $G$ is a map $\sigma:G\to \GL(V)$  satisfying
	\begin{enumerate}
		\item $\sigma(g)\sigma(h)=\alpha(g,h)\sigma(gh)$ for all $g,h\in G$, and 
		\item $\sigma(1)=id$.
	\end{enumerate}
\end{definition}
Notice that $\alpha$ has to be a cocycle in $Z^2(G,\IC^*)$, satisfying $\alpha(g,1)=\alpha(1,g)=1$.

The direct sum of two $\alpha$-twisted representations is again an $\alpha$-twisted representation. We denote by $R_\alpha(G)$ to the Grothendieck group associated to the monoid of isomorphism classes of $\alpha$-twisted representations of $G$.

The tensor product of an $\alpha$-twisted representation with a $\beta$-twisted representation is an $(\alpha+\beta)$-twisted representation. It can be extended to Grothendieck groups obtaining a product
$$R_\alpha(G)\otimes R_\beta(G)\to R_{\alpha+\beta}(G).$$
An interesting property of the groups $R_\alpha(G)$ is the following.
\begin{theorem}\label{cohomologos}
	If $\alpha$ is cohomologous to $\beta$, then there is a bijective correspondence between $\alpha$-twisted representations of $G$ and $\beta$-twisted representations of $G$, and an isomorphism of abelian groups
	$$R_\alpha(G)\cong R_\beta(G).$$
\end{theorem}
More details on projective representations can be found in Section 2 of \cite{dwyer2008} or in \cite{karpilovsky}.

 Now we recall the definition of equivariant twisted K-theory for this kind of twistings. Let $\alpha\in Z^2(G,S^1)$ be a torsion cocycle of order $n$. One can assume that the cocycle takes values in $\IZ/n\IZ\subseteq S^1$. We call such a cocycle \emph{normalized}. Then  we can assign to $\alpha$ a central extension of $G$ by $\IZ/n\IZ$ as follows.
\begin{equation}\label{extension}1\to\IZ/n\IZ\to G_\alpha\xrightarrow{\rho}G,\to1\end{equation}
where $G_\alpha$ as a set is $G\times \IZ/n\IZ$ and if we denote by $\sigma$ a fixed generator of $\IZ/n\IZ$, the product is given by
$$(g,\sigma^j)\cdot (h,\sigma^k)=(gh,\alpha(g,h)\sigma^{j+k}).$$

We have the following theorem. For a proof consult \cite{dwyer2008}.
\begin{theorem}Let $G$ be a finite group and let $\alpha$ be a 2-cocycle taking values in $\IZ/n\IZ$. There is a bijective correspondence between $\alpha$-twisted representations of $G$ and representations of $G_\alpha$ where $\sigma$ acts by multiplication by $e^{2\pi i/n}$
\end{theorem}


\begin{definition}
	Let $G$ be a discrete group  and $\alpha\in Z^2(G,S^1)$ be a normalized cocycle of finite order $n$, consider the central extension \eqref{extension} associated to $\alpha$. Let $X$ be a finite, proper $G$-CW-complex. An $\alpha$-twisted $G$-vector bundle over $X$ is a complex $G_\alpha$-vector bundle $p:E\to X$ such that 
	\begin{enumerate}
		\item The action of $G_\alpha$ on $E$ covers the action of $G$ on $X$. That means, for every $g\in G_\alpha$ and $z\in E$
		$$p(g\cdot z)=\rho(g)\cdot p(z).$$
		\item The action of $\IZ/n\IZ$  on $E$ is given by multiplication by $e^{2\pi i/n}$ on the fibers.
	\end{enumerate}
\end{definition}
If $E$ and $F$ are both $\alpha$-twisted $G$-vector bundles, then so is $E\oplus F$. Then one can consider the monoid of isomorphism classes of $\alpha$-twisted vector bundles with the operation of direct sum. Let ${ }^\alpha K_G(X)$ the associated Grothendieck group. Similarly to the case of twisted representations, ${ }^\alpha K_G(X)$ is not a ring, but we have an external product $${ }^\alpha K_G(X)\otimes{ }^\beta K_G(X)\to { }^{\alpha+\beta} K_G(X),$$
analogous to the product in twisted representations. For details on this product see Theorem 3.4 in  \cite{dwyer2008}.

One crucial property of the above construction of twisted K-theory is the following. Consider an extension as \eqref{extension}, we have an inclusion $${ }^\alpha K_G(X)\to K_{G_\alpha}(X),$$for every finite proper $G$-CW-complex. It allows to extend many results obtained for equivariant K-theory for proper actions to the  twisted case. In particular results from \cite{luck-cherncharacters}.
\begin{definition}[Def. 4.4 in \cite{ademruan}]
	Suppose that $X/G$ (With $G$ a finite group) is a quotient orbifold, and let $\alpha\in Z^2(G;S^1)$ be a torsion 2-cocycle, then one can define its \emph{$\alpha$-twisted orbifold K-theory} just as its $\alpha$-twisted $G$-equivariant K-theory, in other words
	$${ }^\alpha K_{orb}^*(X/G)={ }^\alpha K_G^*(X).$$
\end{definition}
\begin{remark}
	There is a more general definition of twisted K-theory taking twistings as projective unitary stable equivariant bundles, they are classified by third degree integer cohomology classes of $X\times_GEG$ (see \cite{atiyahsegaltwisted} and \cite{BEJU}). However given a torsion cocycle $\alpha\in Z^2(G;S^1)$ it is possible to associate a projective unitary stable equivariant bundle to $\alpha$ in such way that both definition of twisted K-theory are equivalent in this case. For details see Sec. 5.4 in \cite{BEUV}.
\end{remark}

\subsection{The Bredon module of projective representations}
Fix a normalized torsion cocycle $\alpha\in Z^2(G,S^1)$. We define a Bredon module $\calr_\alpha$ over $\OrGF{G}{\calfin}$. On objects its defined by $\calr_\alpha(G/H)=R_\alpha(H)$ and on morphism as follows, if $\phi:G/H\to G/K$ is a morphism in $\OrGF{G}{\calfin}$, $\phi$ is determined by an element $g\in G$ with $gHg^{-1}\subseteq K$ and $\phi(g'H)=g'gK$. We define $\calr_\alpha(\phi):R_\alpha(K)\to R_\alpha(H)$ as the composition
$$R_\alpha(K)\xrightarrow{i^*} R_\alpha(gHg^{-1})\xrightarrow{c_g^*}R_\alpha(H),$$
where the first map is the pullback of the inclusion $i:gHg^{-1}\to K$ and the last one is the pullback of the isomorphism induced by conjugation by $g$. 


\begin{definition}Let $X$ be a proper $G$-CW-complex. The graded group $H^*_G(X;\calr_\alpha)$ is called the $\alpha$-twisted Bredon cohomology of $X$.
\end{definition}
\subsection{Spectral sequence}
There is an spectral sequence constructed by Davis and Luck converging to equivariant K-theory for proper $G$-CW-complexes whose $E_2$-term is given in terms of Bredon cohomology. One can adapt this spectral sequence to be used in the case of twisted K-theory for discrete torsion. For a more general approach the reader can consult \cite{BEUV}.
\begin{theorem}\label{a-h-twisted}
	Let $X$ be a finite $G$-CW-complex and let $\alpha\in Z^2(G,S^1)$ be a normalized torsion cocycle. Then there is a spectral sequence with
	$$E_2^{p,q}=\begin{cases}
	H^p_G(X;\calr_\alpha)&\text{if $q$ is even}\\
	0& \text{if $q$ is odd}
	\end{cases}$$so that $E_\infty^{p,q}\Rightarrow { }^\alpha K_G^{p+q}(X).$
\end{theorem}

\section{Pullback of groups}
Let $\Gamma$ be a group as in diagram \eqref{diagrampullback}. In this section we describe the representation ring of a finite subgroup of $\Gamma$ in terms of the representation rings of finite subgroups of $G$, $H$ and $K$. Using that description and the algebraic Kunneth formula we obtain the main result of the paper.

If the group $\Gamma$ comes from a diagram \eqref{diagrampullback} then it is isomorphic to a subgroup of $G\times H$, namely $\Gamma\cong\{(g,h)\in G\times H\mid \pi_1(g)=\pi_2(h)\}$.  
\subsection{The representation ring of a pullback}
Consider a pullback diagram of finite groups

\begin{equation}\label{pullbackrep}
\xymatrix{
	\Lambda \ar[r]^{p_2}\ar[d]_{p_1} & Q\ar[d]_{\pi_2}\\
	P \ar[r]^{\pi_1} &  S}\end{equation}

If we apply the representation ring functor we obtain the following diagram

\begin{equation}\label{pullbackrepring}
\xymatrix{
	R(\Lambda)  & R(Q)\ar[l]^{p_2^*}\\
	R(P)\ar[u]_{p_1^*}  &  R(S).\ar[l]^{\pi_1^*}\ar[u]_{\pi_2^*}}\end{equation}
We will prove that diagram \eqref{pullbackrepring} is a pushout. In other words we have the following theorem.
\begin{theorem}\label{pullbackrepringiso}
	Let $\Lambda$, $Q$, $P$ and $S$ be finite groups as in diagram \eqref{pullbackrep}. There is a ring isomorphism $$m:R(P)\otimes_{R(S)}R(Q)\to R(\Lambda).$$
\end{theorem}
\begin{proof}
	In order to avoid confusion, in this proof we denote the product on $R(\Lambda)$, $R(P)$ and $R(Q)$ by $\cdot$ \ and the generators of the tensor product by $\rho\otimes\gamma$.
	
	The map $m$ is defined as
	\begin{align*}
	m:R(P)\otimes_{R(S)}R(Q)&\to R(\Lambda)\\
	\rho\otimes\gamma&\mapsto p_1^*(\rho)\cdot p_2^*(\gamma).\end{align*}

		First we prove that the map $m$ is well defined. Let $\xi\in R(S)$, $\rho\in R(P)$ and $\gamma\in R(Q)$, it is enough to prove that the elements 
		$$m(\pi_1^*(\xi)\cdot\rho\otimes\gamma)\text{\  and \ }m(\rho\otimes\pi_2^*(\xi)\cdot\gamma)$$in $R(\Lambda)$ have the same character.
		Let $(g,h)\in\Lambda$
		\begin{align*}
		\charr\big(m(\pi_1^*(\xi)\cdot\rho\otimes\gamma)\big)(g,h)&=\charr\big(p_1^*(\pi_1^*(\xi))\cdot p_1^*(\rho)\cdot p_2^*(\gamma)\big)(g,h)\\
		&=\charr(\pi_1^*(\xi)\cdot\rho)(g)\charr(\gamma)(h)\\
		&=\charr(\xi)(\pi_1(g))\charr(\rho)(g)\charr(\gamma)(h)\\
		&=\charr(\rho)(g)\charr(\xi)(\pi_2(h))\charr(\gamma)(h)\\
		&=\charr\big(m(\rho\otimes\pi_2^*(\xi)\cdot\gamma)\big)(g,h)
		\end{align*}
		Now we will prove that $m$ is an isomorphism.
		Consider the following diagram with exact rows
		$$\xymatrix{
			0\ar[r]&\ker(\pi)\ar[r]\ar[d]^{m_2}& R(P)\otimes_\IZ R(Q)\ar[r]^{\pi}\ar[d]^{m_1}&R(P)\otimes_{R(S)}R(Q)\ar[r]\ar[d]^{m}&0\\
			0\ar[r] &\ker(i^*)\ar[r]&R(P\times Q)\ar[r]^{i^*}&R(\Lambda)\ar[r]&0  
			}$$
		where map $\pi$ is the quotient by all relations defining the tensor product over $R(S)$, the map $i^*$ is the pullback of the inclusion $i:\Lambda\to P\times Q$, the map $m_1$ is the natural isomorphism given by  tensor product over $\IZ$ and the map $m_2$ is the restriction of $m_1$ to $\ker(\pi)$. We will prove that the above diagram is commutative and that $m_2$ is an isomorphism. 
	
	First we need to verify that $m_1(\ker(\pi))\subseteq \ker(i^*)$.
	Let $(g,h)\in \Lambda$,
	\begin{align*}
	\charr\big(i^*&(m_2(\pi_1^*(\xi)\cdot\rho\otimes\gamma-\rho\otimes\pi_2^*(\xi)\cdot\gamma))\big)(g,h)= \\&\charr \big(p_1^*(\pi_1^*(\xi))\cdot p_1^*(\rho)\cdot p_2^*(\gamma)-p_1^*(\rho)\cdot p_2^*(\pi_2^*(\xi))\cdot p_2^*(\gamma)\big)(g,h)=0
	\end{align*}
	
	Now we prove that $\ker(i^*)= m_1(\ker(\pi))$. For this we will prove that if $f$ is a class function in $P\times Q$ such that $i^*(f)\equiv 0$ and $f$ is orthogonal to every character in $m_1(\ker(\pi))$, then $f$ has to be zero.
	
	Suppose that for every $\xi\in R(S)$, $\rho\in R(P)$ and $\gamma \in R(Q)$
	$$\sum_{(g,h)\in P\times Q}\overline{f(g,h)}\charr(\rho)(g)\charr(\gamma)(h)[\charr(\xi)(\pi_2(h))-\charr(\xi)(\pi_1(g))]=0.$$
	Let us fix $\rho\in R(P)$ and let
	$$\eta(g)=\sum_h \overline{f(g,h)}\charr(\gamma)(h)[\charr(\xi)(\pi_2(h))-\charr(\xi)(\pi_1(g))].$$
	We observe that $\eta$ is a class function on $P$ that is orthogonal to every $\rho$ in $R(P)$, then $\eta\equiv0$.
	
	By a similar argument we conclude that for every $(g,h)\in P\times Q$ and $\xi\in R(S)$
	
	\begin{equation}\label{caracter}\overline{f(g,h)}[\charr(\xi)(\pi_2(h))-\charr(\xi)(\pi_1(g))]=0.\end{equation}
	
	We already know that $f(g,h)=0$ if $(g,h)\in\Lambda$, then let $(g,h)\notin \Lambda$, we have two cases. First suppose that $\pi_1(g)$ is conjugate to $\pi_2(h)$ in $S$, in this case there is $\bar{h}\in Q$ such that $(g,\bar{h}h\bar{h}^{-1})\in\Lambda$ and then $f(g,h)=f(g,\bar{h}h\bar{h}^{-1})=0$.
	
	Suppose now that $\pi_1(g)$ is not conjugate to $\pi_2(h)$ in $S$, in this case there is  $\xi\in R(S)$ such that $\charr(\xi)(\pi_1(g))\neq\charr(\xi)(\pi_2(g))$ and equation \eqref{caracter} gives us that $f(g,h)=0$. Then we conclude that $\ker(i^*)=m_1(\ker(\pi))$. The map $m_2$ is an isomorphism because it is the restriction of $m_1$ and since the diagram is commutative we conclude that $m$ is a ring isomorphism.
	\end{proof}
	An analogous result is valid considering the twisted representation group.
	\begin{corollary}
		Let $\Lambda$, $Q$, $P$ and $S$ be finite groups as in the diagram \eqref{pullbackrep}. and let $\alpha\in Z^2(P;\IZ/n\IZ)$ be a normalized torsion cocycle. There is an isomorphism of $R(S)$-modules $$m:R_\alpha(P)\otimes_{R(S)}R(Q)\to R_{p_1^*(\alpha)}(\Lambda).$$
	\end{corollary}

\subsection{Bredon cohomology of pullbacks}


If $X$ is a proper $G$-CW-complex and $Y$ is a proper $H$-CW-complex, the product $X\times Y$ has a natural structure of ($G\times H$)-CW-complex with each cell corresponding to a product of cells of $X$ and $Y$. Since $\Gamma$ is a subgroup of $G\times H$ by Corollary 3.4.4 on \cite{maysid} we can suppose that $X\times Y$ has a structure of $\Gamma$-CW-complex. For this $\Gamma$-space we have the following version of the Eilenberg-Zilber theorem.


\begin{proposition}[Eilenberg-Zilber]\label{eilenberg-zilber}Let $X$ be a proper $G$-CW-complex and $Y$ be a proper $H$-CW-complex. There is a natural equivalence of 
$\OrGF{\Gamma}{\calfin_G\times_K\calfin_H}$-chain complexes  
		$${\underbar{C}}_{*}^{\,\Gamma}(X\times Y)\cong 
		(\underbar{C}_*^{\,G}(X)\otimes\underbar{C}_*^{\,H}(Y))|_\Gamma.
		$$
		
\end{proposition}
\begin{proof}
Let $P\times_{\pi_1(P)}Q$ be an element in $\calfin(G)\times_K\calfin(H)$. It is a well know fact that we have an isomorphism of chain complexes
\begin{align*}
f_{P,Q}:C_*(X^P\times Y^Q)&\rightarrow C_*(X^P)\otimes_\IZ C_*(Y^Q)\\
e\times f& \to e\otimes f.\end{align*}
We only need to verify that it is a natural transformation. Consider $$(p,q)_*:\Gamma/P\times_{\pi_1(P)}Q\to \Gamma/P'\times_{\pi_1(P')}Q'$$ a morphism in $\OrGF{\Gamma}{\calfin(G)\times_K\calfin(H)}.$

We need to prove that the following diagram commutes
$$\xymatrix{C_n((X\times Y)^{P\times_{\pi_1(P)}Q})\ar[rr]^{(p,q)_*}\ar[d]^{f_{P,Q}}&&C_n((X\times Y)^{P'\times_{\pi_1(P')}Q'})\ar[d]^{f_{P',Q'}}\\
	\bigoplus\limits_{i+j=n}C_i(X^P)\otimes_\IZ C_j(Y^Q)\ar[rr]^{\oplus (p_*\otimes q_*)}&&\bigoplus\limits_{i+j=n}C_i(X^{P'})\otimes_\IZ C_j(Y^{Q'})}.$$

But this is clear because
$$(p_*\otimes q_*)\circ f_{P,Q}(e\times f)=(p_*\otimes q_*)(e\otimes f)=pe\otimes qf,$$
and
$$f_{P,Q}\circ((p,q)_*(e\times f))=f_{P,Q}(pe\times qf)=pe\times qf.$$
	\end{proof}
Then we have a natural equivalence of Bredon modules analogous to the isomorphism in Theorem \ref{pullbackrepringiso}.
\begin{theorem}
	There is a natural equivalence of Bredon modules
	$$\calr^G\otimes_{\calr^K}\calr^H\xrightarrow{\bar{m}}\calr^\Gamma.$$
\end{theorem}
\begin{proof}
	The map $\bar{m}$ is defined using the isomorphism in Theorem \ref{pullbackrepringiso}. Let $\Gamma/(P\times_{\pi_1(P)}Q)\in\OrGF{\Gamma}{\calfin_G\times_K\calfin_H}$, we define $$\bar{m}(\Gamma/(P\times_{\pi_1(P)}Q))=m_{P,Q},$$ where $m_{P,Q}$ is the isomorphism in Theorem \ref{pullbackrepringiso} for the group $P\times_{\pi_1(P)}Q$. We only need to verify that this map is natural. Consider a morphism $$(g,h):\Gamma/(P\times_{\pi_1(P)}Q)\to \Gamma/(P'\times_{\pi_1(P')}Q'),$$recall that this morphism is characterized by the condition $$(P\times_{\pi_1(P)}Q)^{(g,h)}\subseteq (P'\times_{\pi_1(P')}Q').$$
	Following the notation in the proof of Theorem \ref{pullbackrepringiso}, we have
	$$(g,h)^*(m(\xi\otimes \gamma))=(g,h)^*(p_1^*(\xi)\cdot p_2^*(\gamma))=\left(p_1^*(\xi)\cdot p_2^*(\gamma)\right)\mid_{(P\times_{\pi_1(P)}Q)^{(g,h)}}.$$
	On the other hand 

	$$m((g^*\otimes h^*)(\xi\otimes \gamma)=m(\xi\mid_{P^g}\otimes\gamma\mid_{Q^h})=p_1^*(\xi\mid_{P^g})\cdot p_2^*(\gamma\mid_{Q^h})$$
	And the last terms of both expressions are the same because pullbacks commute with restrictions on the representation ring. Then the map $\bar{m}$ is natural.
	\end{proof}
For a Bredon module $\calm$ over $\OrGF{G}{\calf}$ one can define the Bredon cohomology groups of $G$ with coefficients in $\calm$ by
$$H^i(G;\calm)=\Ext^i(\underline{\mathbb{Z}},\calm),$$
where $\Ext^i$ is certain functor generalizing the Ext-functor defined for groups to the context of Bredon modules. For details consult \cite{valettemislin} on pages 7-27.

In particular we have the following result
\begin{proposition}There is an isomorphism
	$$H^0(G;\calm)\cong\colim_{G/K\in\OrGF{G}{\calf}}\calm(G/K),$$
	where we are taking the inverse limit with respect to the induced morphism from the orbit category $\OrGF{G}{\calf}$.
\end{proposition}
If $\calm$ is a Green functor it endows every $\calm(G/K)$ with an $H^0(G;\calm)$-module structure.
\begin{theorem}\label{eilenberg-zilber-coh}Let $\calM$, $\calM'$ and $\calN$ be Bredon modules over $G$, $H$ and $K$ respectively, suppose that $\calM$ is a $\pi_1^*\calN$-module and $\calM'$ is a $\pi_2^*\calN$-module.  Then there is an isomorphism of $\OrGF{\Gamma}{\calfin_G\times_{\calfin_K}\calfin_H}$-cochain complexes
	\begin{align*}\Hom_{\OrGF{\Gamma}{\calfin}}((\underbar{C}_*(X)\otimes\underbar{C}_*(Y))\mid_\Gamma,&\calM\otimes_{\calN}\calM')\to\\ &C_G^*(X,\calM)\otimes_{H^0(K,\calN)}C_H^*(Y,\calM').\end{align*}
\end{theorem}
\begin{proof}
	At degree $n$ the left hand side cochain complex is
	$$\bigoplus_{e_\lambda,f_\mu}\Hom_\IZ\left(\IZ[e_\lambda]\otimes\IZ[f_\mu],\calM(G/P_\lambda)\otimes_{\calN(K/\pi_1(P_\lambda))}\calM'(H/Q_\mu)\right).$$
	
	At degree $n$ the right hand side is
	$$\left(\bigoplus_{e_\lambda}\Hom(\IZ[e_\lambda],\calM(G/P_\lambda))\right)\otimes_{H^0(K,\calN)}\left(\bigoplus_{f_\mu}\Hom(\IZ[f_\mu],\calM'(H/Q_\mu))\right).$$
	There is an isomorphism of cochain complexes of $H^0(K;\calN)$-modules from the right hand side to 
	$$\bigoplus_{\lambda,\mu}\Hom_\IZ\left((\IZ[e_\lambda]\otimes\IZ[f_\mu]),\calM(G/P_\lambda)\otimes_{H^0(K,\calN)}\calM'(H/Q_\mu)\right).$$
	For every $P_\lambda\times_{\pi_1(P_\lambda)}Q_\mu\in\calfin_G\times_K\calfin_H$ we have that the $H^0(K;\caln)$-module structure of $M(G/P_\lambda),$ respectively of $M'(H/Q_\mu)$, factors through the sequences
	$$H^0(K;\caln)\to N(K/\pi_1(P_\lambda))\to M(G/P_\lambda) $$
	and$$H^0(K;\caln)\to N(K/\pi_1(P_\lambda))\to M'(H/Q_\mu).$$
	  Then the $H^0(K;\calN)$-module $$\calM(G/P_\lambda)\otimes_{H^0(K,\calN)}\calM'(H/Q_\mu)$$ is isomorphic to $$\calM(G/P_\lambda)\otimes_{\calN(K/\pi_1(P_\lambda))}\calM'(H/Q_\mu).$$ Given that  isomorphism is compatible with the coboundary map, then we have an isomorphism of cochain complexes. 
	
	\end{proof}
	In order to obtain a Kunneth formula for Bredon cohomology of a $\Gamma$-CW-complex $X\times Y$ we need to recall the algebraic Kunneth formula, for a proof see for example \cite[Thm. 3.6.3]{weibel}.
	\begin{lemma}\label{kunnethalg}
		Let $R$ be a commutative ring with unity and let $(C,d)$ and $(C',d')$ be $R$-cochain complexes, if $C_n$ and $d(C_n)$ are flat for each $n$ then there is a split exact sequence
		\begin{align*}0\to \bigoplus\limits_{p+q=n}H^p(C)\otimes_RH^q(C')&\to H^n(C\otimes_RC')\to\\& \bigoplus\limits_{p+q=n+1}\Tor^R(H^p(C),H^q(C')).\end{align*} 
	\end{lemma}

	Finally we can introduce our main result.
	
	\begin{theorem}[Kunneth formula for Bredon cohomology of pullbacks]\label{maintheorem}Let $\Gamma$ be a group coming from a diagram \eqref{diagrampullback} and suppose that $(\calf,\calf_1,\calf_2)$ is a sequence of compatible families of subgroups of $\Gamma$. Let $\calN$ be a Green functor over $\OrGF{\calf}{K}$, $\calM$ be a $\pi_1^*\calN$-module over $\OrGF{\calf_1}{G}$ and $\calM'$ be a $\pi_2^*\calN$-module over $\OrGF{\calf_2}{H}$.
		Let $X$ be a $\calf_1$-CW-complex and $Y$ be a $\calf_2$-CW-complex. Then if we consider $X\times Y$ as a $(\calf_1\times_\calf\calf_2)$-CW-complex there is a split exact sequence
		\begin{align*}0\to \bigoplus\limits_{p+q=n}H^p_G(X;\calM)\otimes&_{H^0(K;\calN)}H^q_H(Y;\calM')\to H^n_\Gamma(X\times Y;\calM\otimes_{\calN}\calM')\\\to& \bigoplus\limits_{p+q=n+1}\Tor^{H^0(K;\calN)}(H^p_G(X;\calM),H^q_H(Y;\calM')).\end{align*} 
	\end{theorem}
    \begin{proof}
    	It is a immediate consequence of Theorems \ref{eilenberg-zilber}, \ref{eilenberg-zilber-coh} and Lemma \ref{kunnethalg}.
    	\end{proof}

 	\section{Examples and applications}
 	In this section we apply Theorem \ref{maintheorem} in different cases. In particular, we compute Bredon cohomology with coefficients in representations of the classifying space for proper actions for important examples studied in \cite{vafa-witten}. As a consequence we obtain a computation of twisted orbifold K-theory ${ }^\alpha K^*(T^6/(\IZ/2\IZ)^2)$ for some discrete twisting $\alpha$. 	
\begin{remark}Theorem \ref{maintheorem} has the following consequences:
 \begin{enumerate}
 \item Consider the direct product of two discrete groups $G\times H$ as a pullback over the trivial group acting over a $(G\times H)$-CW-complex. Applying Theorem \ref{maintheorem} we obtain the well known Kunneth formula for direct product of groups. In this case the tensor product is taken over $\IZ$.
 \item Consider the following pullback diagram over a finite group $G$
 \begin{equation*}\label{pullbackrepG}
 		\xymatrix{
 			G \ar[r]^{id}\ar[d]_{id} & G\ar[d]_{id}\\
 			G \ar[r]^{id} &  G}\end{equation*}
 If we apply Theorem \ref{maintheorem} to the above pullback we obtain a Kunneth formula for Bredon cohomology of the product of two $G$-spaces with the diagonal $G$-action. In this case the tensor product is taken over $\calN(G)$.
 \end{enumerate}
\end{remark} 
 	On the other hand if we consider the trivial family and the constant Bredon module $\underbar{R}$ (where $R$ is a commutative ring with unity) with the $G$-space $EG$ and the $H$-space $EH$ we obtain the following Kunneth formula for group cohomology of pullbacks
 	\begin{theorem}[Kunneth formula for group cohomology of pullbacks] Let $R$ be a PID. There is a split exact sequence
    \begin{align*}0\to \bigoplus\limits_{p+q=n}H^p(G;\underbar{R})\otimes_{H^0(K;\underline{R})}&H^q(H;\underline{R})\to H^n(\Gamma;\underline{R})\to\\& \bigoplus\limits_{p+q=n+1}\Tor^{H^0(K;\underline{R})}(H^p(G;\underbar{R}),H^q(H;\underbar{R})).\end{align*}
 	\end{theorem}
 	\begin{proof}
 	Notice that the $\Gamma$-space $EG\times EH$ is a model for $E\Gamma$. Applying Theorem \ref{maintheorem} to the $\Gamma$-CW-complex $EG\times EH$ we obtain the result.
 		\end{proof}
 Consider the Bredon modules $\calr$ and $\calr_\alpha$ defined on the family of finite subgroups. Applying Theorems \ref{pullbackrepringiso} and \ref{maintheorem} we obtain the following theorem.	
 	\begin{theorem}[Kunneth formula for Bredon cohomology of proper actions with coefficients in representations]\label{kunnethformula}Let $X$ be a proper $G$-CW-complex, let $Y$ be a proper $H$-CW-complex. Let $\alpha\in Z^2(G;\IZ/n\IZ)$ be a normalized torsion cocycle of $G$.  There is a split exact sequence
 		\begin{align*}0\to \bigoplus\limits_{p+q=n}H^p_G(X;\mathcal{R}_\alpha)\otimes&_{H^0(K;\mathcal{R})}H^q_H(Y;\mathcal{R})\to H^n_\Gamma(X\times Y;\mathcal{R}_{p_1^*(\alpha)})\to\\& \bigoplus\limits_{p+q=n+1}\Tor^{H^0(K;\mathcal{R})}(H^p_G(X;\mathcal{R}_\alpha),H^q_H(Y;\mathcal{R})).\end{align*}When the group $K$ is finite one can substitute $H^0(K; \mathcal{R})$ in the above formula by the representation ring $R(K)$.
 	\end{theorem}
 \subsection{The Vafa-Witten groups}
 Here we apply the above results to compute the Bredon cohomology with coefficients in twisted representations for the classifying space for proper actions of the Vafa-Witten groups defined in \cite{vafa-witten}. Using that computation and the Atiyah Hirzebruch spectral sequence we obtain some computations of twisted and untwisted orbifold K-theory groups  associated to the action of these groups on $\IR^6$.
 
  From now on, the cyclic group $\IZ/4\IZ$ will be seen as the set of $4$-th roots of unity generated by $\omega = e^{\frac{2\pi i}{4}}$.
\subsubsection{The group $\IZ^6\rtimes\IZ/4\IZ$}
Consider the action of $\IZ/4\IZ$ on $\IZ^{6}$ induced from the 
action of $\IZ/4\IZ$ on $\IC^3$, given by
$$
k(z_1,z_2,z_3)=(-z_1,iz_2,iz_3).
$$
We denote the $\IZ^6$ with the above $\IZ/4\IZ$-action by $M$. Now we take the semidirect product 
\begin{align*}
1\to M \to M\rtimes\IZ/4\IZ \to \IZ/4\IZ\to 1.
\end{align*}

Note that we have the following decomposition of $\IZ/4\IZ$-modules
$$M=(M_1)^2\oplus (M_2)^2,$$
where $M_1$ has rank one and the generator acts by multiplication by $-1$; $M_2$ has rank two and the generator acts by multiplication by $i$.

This decomposition induces  a decomposition of the group $M\rtimes\IZ/4\IZ$ as the multiple pullback

\begin{equation}\label{multiplepullback}
(M_1\rtimes\IZ/4\IZ)\times_{\IZ/4\IZ}(M_1\rtimes\IZ/4\IZ)\times_{\IZ/4\IZ}(M_2\rtimes\IZ/4\IZ)\times_{\IZ/4\IZ}(M_2\rtimes\IZ/4\IZ).
\end{equation}
\begin{remark}
	If we divide $\IC^3$ by $M$, we have an action of $\IZ/4\IZ$ over $T^6$. Let $Y_1=T^6/(\IZ/4\IZ)$ the associated orbifold.
\end{remark}
For simplicity we denote $M_1\rtimes\IZ/4\IZ$ by $G$ and $M_2\rtimes\IZ/4\IZ$ by $H$.

Consider $\IR$ with the $G$-action defined as follows. Let $(n,\omega)\in G$ and $x\in\IR$
$$(n,\omega)\cdot x=n+\omega^2x.$$

 A $G$-CW-complex decomposition for $\IR$ is given by
 \begin{center}
 	\begin{tabular}{|ccccc|}	\hline
 		&&0-cells&& \\ \hline
 		&&&&\\
 		$v_1$ & $O$ & $(0,\bar{1})$ & $\IZ/4\IZ$&\\
 		$v_2$ & $P$ & $ (1,\bar{1})$& $\IZ/4\IZ$&\\
 		\hline 		
 		&&1-cell&&\\ \hline
 		&&&&\\
		$e_1$ & $OP$ & $(0,\bar{2})$ & $\IZ/2\IZ$&\\ 		
 		\hline 		
 	\end{tabular}
 \end{center}
 
 The first column is an enumeration of equivalence classes of cells; the second lists
 a representative of each class; the third column gives a generating element for the
 stabilizer of the given representative; and the last one is the isomorphism type of the stabilizer. The inclusions $i_j:stab(e_1)\to stab(v_i)$ ($i=1,2$) are given by multiplication by 2.
  
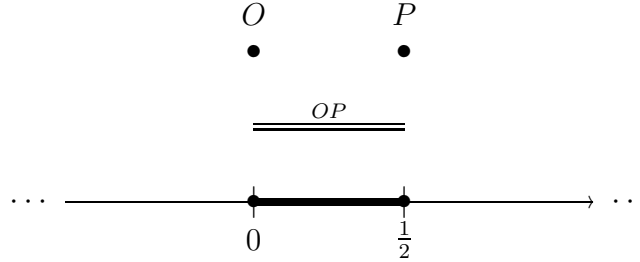
\[\xy 
  (-40,0)*+{\cdots};
  (40,0)*+{\cdots};
  (-10,-5)*+{0};(-10,0)*+{|};
  (10,-5)*+{\frac{1}{2}};(10,0)*+{|};
  (-10,0)*+{\bullet}; (10,0)*+{\bullet};
  (-10,20)*+{\bullet}; (10,20)*+{\bullet};
  (-10,25)*+{O}; (10,25)*+{P};
  \ar @{<-} (-35,0);(35,0);
  \ar @{=} (-10,0);(10,0);       
  \ar @{=}^{OP} (-10,10);(10,10)        
  \endxy\] 
 Recall that $R(\IZ/4\IZ)\cong \IZ[\zeta]/\langle\zeta^4-1\rangle$. We use the following notations
 \begin{enumerate}
 	\item $I=\langle\zeta^2-1\rangle$,
 	\item $J=\langle\zeta^2+1\rangle$,
 	\item $K=\langle \zeta +1\rangle$,
 	\item $L=\langle\zeta-1\rangle$.
 \end{enumerate}  
 The Bredon cochain complex of $\IR$ with coefficients in $\calr$ has the form
 $$0\to \calr(G/stab(v_1))\oplus \calr(G/stab(v_2))\xrightarrow{\delta} \calr(G/stab(e_1))\to 0.$$
 The map $\delta$ is $i_1^*-i_2^*$. Note that $\delta$ is surjective therefore
 $$H^1_G(\IR;\calr)=0.$$
 On the other hand $$H^0_G(\IR;\calr)=\ker(\delta).$$
 
 Note that the Bredon cochain complex is isomorphic to
 $$0\to R(\IZ/4\IZ)^{2}\xrightarrow{\delta} R(\IZ/4\IZ)/I\to 0$$

 
 On the above complex one can see that $\ker(\delta)$ is isomorphic as a $R(\IZ/4\IZ)$-module to $R(\IZ/4\IZ)\oplus I$. Note that this module is projective because $I\oplus J=2R(\IZ/4\IZ)\cong  R(\IZ/4\IZ)$.
 
 Now let us consider the group $H$. Let $(n,m,\omega)\in H$ and $z\in\IC$, define a $H$-action on $\IC$ as follows
 $$(n,m,\omega)\cdot z=(n+im)+\omega z.$$

 A $H$-CW-complex decomposition for $\IC$ is given by
 
 \begin{center}
 	\begin{tabular}{|cccc|}
 		\hline
 		&&0-cells  &   \\
 		\hline &&&\\
 		$a_1$ & $Q$ & $(0,0,\bar{1})$ & $\IZ/4\IZ$ \\ 		
 		$a_2$ & $R$ & $ (-1,0,\bar{3})$& $\IZ/2\IZ$\\
 		$a_3$ & $S$ & $(0,1,\bar{1})$& $\IZ/4\IZ$\\
 		\hline
 		&&1-cells &\\
 		\hline &&&\\
 		$b_1$ & $QR$ & $(0,0,\bar{0})$ & $0$\\
 		$b_2$ & $RS$ & $(0,0,\bar{0})$ & $0$\\
 		\hline
 		&& 2-cell &\\ \hline &&&\\
 		$T$ & $QRSR'$& $(0,0,\bar{0})$ & $0$\\\hline
 	\end{tabular}
 	
 \end{center}
 With similar conventions to the case of the action of $G$ over $\IR$. 
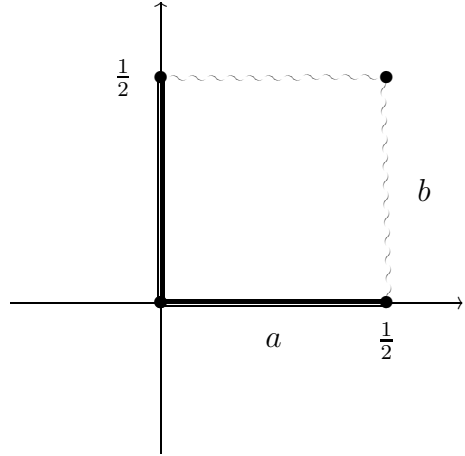
 \[\xy 
 (0,0)*+{\bullet}; (30,0)*+{\bullet}; (0,30)*+{\bullet}; (30,30)*+{\bullet};
 (0,0)*+{\bullet}; (30,-5)*+{\frac{1}{2}}; (-5,30)*+{\frac{1}{2}}; (15,-5)*+{a};(35,15)*+{b};
 \ar @{=} (0,0);(30,0);       
 \ar @{=} (0,0);(0,30); 
 \ar @{~} (0,30);(30,30);             
 \ar @{~} (30,0);(30,30);       
 \ar @{<-} (-20,0);(40,0);       
 \ar @{<-} (0,-20);(0,40);
 \endxy\] 
 
 In this case the boundary is given by $\partial T=0$, $\partial b_1=a_2-a_1$ and $\partial b_2=a_3-a_2$.
 We denote $i_{l,m}:stab(b_m)\to stab(a_l)$ the inclusion.
 
 The Bredon cochain complex with coefficients in representations has the form
 \begin{align*}
 0\to \calr&(H/stab(a_1))\oplus \calr(H/stab(a_2))\oplus \calr(H/stab(a_3))\\&\xrightarrow{\epsilon} \calr(H/stab(b_1))\oplus \calr(H/stab(b_2))\xrightarrow{0} \calr(H/stab(T))\to0.\end{align*}
 Here, the map  $\epsilon$ is defined as follows. Let $\rho_i\in\calr(H/stab(a_i))$, then
 $$\epsilon(\rho_1,\rho_2,\rho_3)=\left(i_{1,1}^*(\rho_1)-i_{3,1}^*(\rho_3), i_{3,2}^*(\rho_3)-i_{2,2}^*(\rho_2)\right).$$
 
 The map $\epsilon$ is surjective, then $H^1_H(\IC;\calr)=0$. On the other hand $H^0_H(\IC;\calr)=\ker{\epsilon}$. To determine $\ker(\delta)$ note that the map $\epsilon$ can be identified with a map
 \begin{align*}R(\IZ/4\IZ)^{2} \oplus R(\IZ/4\IZ)/&I\to \left(R(\IZ/4\IZ)/J\right)^{2}\end{align*}
 of $R(\IZ/4\IZ)$-modules mapping
 \begin{align*}
 (\zeta,0,0)&\mapsto (1,0),\\
 (0,\zeta,0)&\mapsto (0,-1),\\
 (0,0,\zeta)&\mapsto (-1,1).\\
 \end{align*} 

 
 The kernel of the above map is isomorphic as an $R(\IZ/4\IZ)$-module to $$L\oplus L/I\oplus R(\IZ/4\IZ).$$
 Note that the above $R(\IZ/4\IZ)$-module is projective because we have the following isomorphisms of $R(\IZ/4\IZ)$-modules 
 \begin{align*}L/I\oplus K/I\cong R(\IZ/4\IZ)/I, \text{\quad and}\\R(\IZ/4\IZ)/I\oplus R(\IZ/4\IZ)/J\cong R(\IZ/4\IZ).\end{align*}
Then we have proved the following result.
\begin{theorem}\label{partes}Let $H$ and $G$ be the groups defined above. There is an isomorphism of $R(\IZ/4\IZ)/I$-modules
	\begin{itemize}
	  \item 
	  	$H^n_G(\IR;\calr)\cong \begin{cases}
	  	R(\IZ/4\IZ)\oplus I&\text{ for $n=0$}\\
	  	0&\text{ for $n>0$}.
	  \end{cases}$
	  \item $H^n_H(\IC;\calr)\cong \begin{cases}
	  L\oplus (L/I)\oplus R(\IZ/4\IZ) &\text{for $n=0$}\\
	  \IZ &\text{for $n=2$}\\
	  0&\text{for $n\neq0,2$}.
	  \end{cases}$
	\end{itemize}
	Moreover all modules are projective.
\end{theorem}
Since all the above modules are projective, in the corresponding Kunneth formula the term with $\Tor$ is zero and then we have a complete calculation of the Bredon cohomology groups of $\underbar{E}\Gamma$ with coefficients in representations.

\begin{theorem}\label{vafa-witten-ex}
	Let $H$ and $G$ be the groups in Theorem \ref{partes}. There is an isomorphism of $\IZ$-graded, $R(\IZ/4\IZ)$-modules
	\begin{align*}H^*_{M\rtimes\IZ/4\IZ}(\IC^3;\calr)\cong H^*_G(\IR;\calr)&\otimes_{R(\IZ/4\IZ)}H^*_G( \IR;\calr)\otimes_{R(\IZ/4\IZ)}\\&H^*_H(\IC;\calr)\otimes_{R(\IZ/4\IZ)}H^*_H(\IC;\calr).\end{align*}
\end{theorem}
\begin{proof}Bredon cohomology of $\IR$ and $\IC$ are projective $R( \IZ/4\IZ)$-modules, thus the result follows by applying Theorem \ref{kunnethformula}.
	\end{proof}

	
Since the Bredon cohomology groups of $\IC^3$ are concentrated at even degree, the Atiyah-Hirzebruch spectral sequence of Theorem \ref{a-h-twisted}  collapses at level 2 and we have also a complete calculation of $(M\rtimes\IZ/4\IZ)$-equivariant K-theory of $\IC^3$, it is  the same as the orbifold K-theory of $Y_1$.
\begin{corollary}\label{y1}
	 There is an isomorphism of $\IZ$-graded, $R(\IZ/4\IZ)$-modules
	$$K_{orb}^*(T^6/(\IZ/4\IZ))=K_{M\rtimes\IZ/4\IZ}^*(\IC^3)\cong H^*_{M\rtimes\IZ/4\IZ}(\IC^3;\calr).$$
\end{corollary}
\begin{remark}\label{limited}
	For any non trivial torsion cocycle $\beta \in Z^2(\Gamma;\IZ/n\IZ)$ our method cannot to be used to compute $H^*_{M\rtimes\IZ/4\IZ}(\IC^3;\calr_\beta)$ because from calculations contained in \cite{adempan} we know
	$$H^3({M}_1\rtimes \IZ/4\IZ;\IZ)=0\text{ and }H^3({M}_2\rtimes \IZ/4\IZ;\IZ)=0.$$Then it is not possible to obtain a non-trivial 2-cocycle of $M\rtimes\IZ/4\IZ$ coming from non trivial 2-cocycles of $G$ or $H$.
\end{remark}
\subsubsection{The group $\IZ^6\rtimes(\IZ/2\IZ)^2$}
Consider the action of $(\IZ/2\IZ)^2$ on $\IC^3$, given on generators by
$$\sigma_1(z_1,z_2,z_3)=(-z_1,-z_2,z_3),\hspace{0.5cm}\sigma_2(z_1,z_2,z_3)=(-z_1,z_2,-z_3).$$
We denote $\IZ^6$ with the induced action of $(\IZ/2\IZ)^2$ by $N$. Using the above action we define the semidirect product $N\rtimes(\IZ/2\IZ)^2$.

Note that we have the following decomposition of $(\IZ/2\IZ)^2$-modules:
$$N=(N_1)^2\oplus(N_2)^2\oplus(N_3)^2,$$
where $N_1$ has rank one and $\sigma_i$, $i=1,2$, act by multiplication by $-1$; $N_2$ has rank one, $\sigma_1$ acts by multiplication by -1; $\sigma_2$ acts by multiplication by 1; and $N_3$ has rank one and $\sigma_1$ acts by multiplication by 1 and $\sigma_2$ acts by multiplication by -1. 

That decomposition gives us  a decomposition of the group $N\rtimes(\IZ/2\IZ)^2$ as the multiple pullback of six groups over $(\IZ/2\IZ)^2$: two copies of each $N_j\rtimes (\IZ/2\IZ)^2$, $j=1,2,3$. 
\begin{remark}
	If we take a quotient  by $\IZ^6$, we have an action of $(\IZ/2\IZ)^2$ over $T^6$. Let $Y_2=T^6/(\IZ/2\IZ)^2$ the associated orbifold.
\end{remark}
We will construct a nontrivial torsion cocycle $$\beta\in Z^2(N\rtimes(\IZ/2\IZ)^2;\IZ/2)$$coming from a cocycle of $N_1\rtimes(\IZ/2\IZ)^2$. Consider the non-trivial central extension
$$0\to \IZ/2\IZ\to D_8\to(\IZ/2\IZ)^2\to0.$$With associated 2-cocycle $\alpha\in Z^2(G;S^1)$. To the above extension one can associate a non trivial central extension
\begin{equation}\label{cocycleno}0\to \IZ/2\IZ\to N_1\rtimes D_8\to N_1\rtimes(\IZ/2\IZ)^2\to0.\end{equation}
We denote by $\beta_1\in Z^2(N_1\rtimes(\IZ/2\IZ)^2;\IZ/2\IZ)$ the cocycle associated to extension \eqref{cocycleno}, and by $\beta\in Z^2(N\rtimes(\IZ/2\IZ)^2;\IZ/2\IZ)$ to the pullback of $\beta_1$ on the group $N\rtimes(\IZ/2\IZ)^2$. 

Now we compute $H^*_{N_1\rtimes(\IZ/2\IZ)^2}(\IR;\calr_{\beta_1})$.

Recall that $$R((\IZ/2\IZ)^2)=\IZ[\gamma_1,\gamma_2]/\langle\gamma_1^2-1,\gamma_2^2-1\rangle.$$

Consider $\IR$ with the action of $N_1\rtimes(\IZ/2\IZ)^2$ defined as follows:
$$(n,\sigma_i)\cdot x=n-x\text{ for $i=1,2.$}$$

A $N_1\rtimes(\IZ/2\IZ)^2$-CW-complex decomposition for $\IR$ is given by

\begin{center}
	\begin{tabular}{ccccccccc}
		\hline
		&0-cells  & & & &&1-cell\\
		\hline\\
		$a_1$ & $P$ & $(0,\sigma_1),(0,\sigma_2)$ & $(\IZ/2\IZ)^2$& &$b_1$ & $PQ$ & $(0,\sigma_1\sigma_2)$ & $\IZ/2\IZ$\\
		
		$a_2$ & $Q$ & $ (-1,\sigma_1),(-1,\sigma_2)$& \\
	
		\hline
	
	\end{tabular}
	
\end{center}

The Bredon cochain complex of $\IR$ with coefficients in projective $\beta_1$-representations is isomorphic to (here $\beta_1\mid$ denotes the restriction of $\beta_1$ to the corresponding subgroup)

$$0\to R_{\beta_1\mid}((\IZ/2\IZ)^2)\oplus R_{\beta_1\mid}((\IZ/2\IZ)^2)\xrightarrow{\delta}R_{\beta_1\mid}(\IZ/2\IZ)\to0.$$
The above complex can be viewed as a subcomplex of
$$R(D_8)\oplus R(D_8)\xrightarrow{\bar{\delta}}R(\IZ/4\IZ)\to0.$$
Making an easy calculation on the above complex we obtain
$$
H^p_{N_1\rtimes(\IZ/2\IZ)^2}(\IR;\calr_{\beta_1})=\begin{cases}
\IZ &\text{if $p=0,1$}\\
0&\text{if $p>1$.} 
\end{cases} $$
In order to obtain $$H^*_{N\rtimes(\IZ/2\IZ)^2}(\IC^3;\calr_\beta)$$ we need to calculate $$H^*_{N_j\rtimes(\IZ/2\IZ)^2}(\IR;\calr)$$ for $j=1,2,3$.

First notice that the groups $N_j\rtimes(\IZ/2\IZ)^2$ are isomorphic for $j=1,2,3$ and then it suffices to calculate for $j=1$.

In this case the Bredon cochain complex takes the form 

$$0\to R((\IZ/2\IZ)^2)\oplus R((\IZ/2\IZ)^2)\xrightarrow{\delta}R(\IZ/2\IZ)\to0.$$

We have an isomorphism of projective $R((\IZ/2\IZ)^2)$-modules

$$H^p_{N_1\rtimes(\IZ/2\IZ)^2}(\IR;\calr)\cong\begin{cases}R((\IZ/2\IZ)^2)\oplus I&\text{if }p=0\\
0& \text{if }p>0.\end{cases}$$
where $I=\langle(\gamma_1-\gamma_2,0)\rangle$. Since all above modules are projective, in the corresponding Kunneth formula the term with $\Tor$ is zero and then using Theorem \ref{kunnethformula} we have a complete calculation of the Bredon cohomology groups of $\IC^3$ with coefficients in $\beta$-representations.
\begin{theorem}\label{twistedy2}
	There is an isomorphism of $\IZ$-graded, $R((\IZ/2\IZ)^2)$-modules
	\begin{align*}&H^n_{N\rtimes(\IZ/2\IZ)^2}(\IC^3;\calr_\beta)\cong\\&
	\bigoplus_{\sum_{i=1}^{6}p_i=n}H^{p_1}_{N_1\rtimes(\IZ/2\IZ)^2}(\IR;\calr_{\beta_1})\otimes\left(\bigotimes_{i=2}^6H^{p_i}_{N_1\rtimes(\IZ/2\IZ)^2}(\IR;\calr)\right).\end{align*}
\end{theorem}
Finally, because as all above modules are free over $\IZ$ by Proposition 5.8 in \cite{luck-cherncharacters} we have the following result

\begin{corollary}
There is an isomorphism of $\IZ$-graded, $R((\IZ/2\IZ)^2)$-modules
	\begin{align*}{ }^\alpha K_{orb}(T^6/(\IZ/2\IZ)^2)&= { }^\beta K_{N\rtimes(\IZ/2\IZ)^2}^*(\IC^3)\\&\cong H^*_{N\rtimes(\IZ/2\IZ)^2}(\IC^3;\calr_{\beta}).\end{align*}
\end{corollary}
\begin{remark}
	Although we use Theorem \ref{kunnethformula} to compute twisted K-theory for a very particular twisting, one could follow the same strategy in order to compute $${ }^\alpha K_{orb}^*(T^6/(\IZ/2\IZ)^2)$$for every $\alpha$ coming from 2-cocycles of $G$ or $H$. 
\end{remark}
\bibliographystyle{alpha}
\bibliography{pullback}
\end{document}